\documentclass[12pt]{amsart}
\usepackage{amsmath,amssymb,amsthm,amsfonts}
\usepackage{latexsym}
\usepackage{graphicx,color}
\def\id{\mbox{id}}

\def\inj{\mbox{inj}}

\def\R{\mathbb{R}}

\def\pd#1#2{\frac{\partial #1}{\partial #2}}
\def\vv<#1>{\langle#1\rangle}

\def\XXint#1#2{\setbox0=\hbox{$#1{#2}{\int}$}{#2}\kern-.5\wd0 }

\def\XXint#1#2#3{{\setbox0=\hbox{$#1{#2#3}{\int}$}
     \vcenter{\hbox{$#2#3$}}\kern-.5\wd0}}

\def\vv<#1>{\langle#1\rangle}

\newtheorem{theorem}{Theorem}[section]

\theoremstyle{definition}

\theoremstyle{remark}

\newtheorem{remark}{Remark}[section]

\numberwithin{equation}{section}
\begin{document}
\title{Negativity of Perelman's Li-Yau-Hamilton type expression }
\author{Chengjie Yu}
\address{Department of Mathematics, Shantou University, Shantou, Guangdong, P.R.China}
\email{cjyu@stu.edu.cn}
\maketitle
\markboth{Negativity of $v$}{Chengjie Yu}
\begin{abstract}
In \cite{CYT}, Chau-Tam-Yu proved the non-positivity of Perelman's new Li-Yau-Hamilton type expression $v$ on
noncompact manifolds. In this article, we further prove that $v$ is negative if the Ricci flow is not end up with an Euclidean space.
\end{abstract}

Let $M^n$ be a smooth manifold of dimension $n$ and let $g(t)$ with
$t\in [0,T]$ be a complete solution to the Ricci flow
\begin{equation}
\pd{g_{ij}}{t}=-2R_{ij}
\end{equation}
on $M^n$. We assume that $g(t)$ satisfies Shi's estimate (Ref. Shi \cite{Shi-Deforming metrics})
\begin{equation}\label{eqn-shi-estimate}
\|\nabla^k Rm\|^2\leq \frac{C_k}{t^{k}}
\end{equation}
all over $M\times [0,T]$, for any nonnegative integer $k$, where
$C_k$ is a positive constant depending on $k$. When $M$ is a compact
manifold, this assumption is superfluous since any solution to the
Ricci flow on a compact manifold will automatically satisfy Shi's
estimate (\ref{eqn-shi-estimate}) by Hamilton's work on Ricci flow
on compact manifolds. By the uniqueness result of Chen-Zhu \cite{Chen-Zhu}, such a solution
to the Ricci flow is uniquely determined by its initial metric $g(0)$.

Let $\Box$ and $\Box^*$ be the heat operator and conjugate heat operator along the $g(t)$ respectively. That is,
\begin{equation}
\Box u=\pd{u}{t}-\Delta u\ \mbox{and}
\end{equation}
\begin{equation}
\Box^* u=-\pd{u}{t}-\Delta u+Ru.
\end{equation}
Let $u$ be a fundamental solution of the conjugate heat equation with $u(x,T)=\delta_p$. Let $f$ be a function
on $M\times [0,T)$ such that
\begin{equation}
u(x,t)=(4\pi(T-t))^{-\frac{n}{2}}e^{-f}.
\end{equation}
The new Li-Yau-Hamilton type expression $v$ discovered by Perelman \cite{P1} is
\begin{equation}
v=[(T-t)(2\Delta f-\|\nabla f\|^2+R)+f-n]u.
\end{equation}

The following non-positivity of $v$ was obtained by Perelman \cite{P1} on compact manifolds and by Chau-Tam-Yu \cite{CYT}
on noncompact manifolds.

\begin{theorem}[Non-positivity of $v$] $v\leq 0$ all over $M\times [0,T)$.
\end{theorem}

Using strong maximum principle, we can further get the following negativity of $v$.
\begin{theorem}[Negativity of $v$]\label{thm-negativity-v}If $(M^n,g(T))$ is not the Euclidean space $\R^n$, $v<0$ on $M\times [0,T)$.
\end{theorem}
\begin{proof} We proved it by contradictions. Suppose there is a $(y,s)\in M\times [0,T)$ such that $ v(y,s) = 0$.
Note that $v$ is a sub-solution to the conjugate heat equation and that $$v\leq 0$$ on $M\times [0,T)$. By the
usual strong maximum principle,  $$v(x,t)=0$$ for any $(x,t)\in M\times [s,T)$.

We come to get a contradiction. By the evolution equation of $v$ (Ref. Perelman \cite{P1}),
\begin{equation*}
R_{ij}+f_{ij}-\frac{1}{2(T-t)}g_{ij}=0
\end{equation*}
for any $t\in [s,T)$. Let $X(t)=\frac{T-s}{T-t}\nabla^sf(s)$ for $t\in [s,T)$. Note that, by
the last equation, and Shi's estimate (\ref{eqn-shi-estimate}), $f_{ij}(s)$ is bounded. So $\|X\|$ is
at most linear growth on $M\times [s,T)$. Hence, $X$ can be integrated up on $[s,T)$. Let $\psi_t$ with
$t\in [s,T)$ be the family of diffeomorphisms generated by $X$ ($\psi_s=\id$). Let
\begin{equation*}
\tilde g(t)=\frac{T-t}{T-s}\psi_t^{*}g(s).
\end{equation*}
Then
\begin{equation*}
\begin{split}
\tilde g'(t)=&-\frac{1}{T-s}\psi_t^{*}g(s)+\frac{T-t}{T-s}\psi_t^{*}\mathcal{L}_{X_t}g(s)\\
            =&-\frac{1}{T-s}\psi_t^{*}g(s)+\psi^*_{t}\mathcal{L}_{\nabla^sf(s)}g(s)\\
            =&-\frac{1}{T-s}\psi_t^{*}g(s)+2\psi^*_{t}f_{ij}(s)\\
            =&-2\psi^*_{t}R_{ij}(s)=-2R_{ij}(\tilde g(t)).
\end{split}
\end{equation*}
So $\tilde g$ is a complete solution to the Ricci flow on $[s,T)$ with $\tilde g(s)=g(s)$ and bounded curvatures. By uniqueness of result of Chen-Zhu \cite{Chen-Zhu},
\begin{equation}\label{eqn-soliton}
\frac{T-t}{T-s}\psi_t^{*}g(s)=\tilde g(t)=g(t),\ \mbox{and}
\end{equation}
\begin{equation*}
\|Rm(g(s))\|=\frac{T-t}{T-s}\big(\psi_t^{-1}\big)^{*}\|Rm(g(t))\|\leq \frac{C(T-t)}{T-s},
\end{equation*}
for any $t\in [s,T)$.

Letting $t\to T^-$, we get that $g(s)$ is flat.

Furthermore, by the Gaussian upper bound of fundamental solutions
(Theorem 5.1 in Chau-Tam-Yu \cite{CYT}),
\begin{equation*}
f(x,s)\geq c_1r^2_{0}(p,x)-C_2
\end{equation*}
for any $x\in M$. So, $f(s)$ achieves its minimum at some point $q$. Then,
\begin{equation*}
X_t(q)=\frac{T-s}{T-t}[\nabla^sf(s)](q)=0
\end{equation*}
for any $t\in [s,T)$ and $\psi_t(q)=q$ for any $t\in [s, T)$. By equation (\ref{eqn-soliton}),
\begin{equation}
\inj_q(g(s))=\sqrt{\frac{T-s}{T-t}}\cdot\inj_q(g(t)).
\end{equation}
Letting $t\to T^-$, we have that $\inj_q(g(s))=\infty$. Therefore, $(M,g(s))$ is the standard
Euclidean space and hence $(M, g(T))$ is the standard Euclidean space which is ruled out by
our assumption.
\end{proof}
\begin{remark}
\begin{itemize}
\item[(1)] The argument of the proof is not new. For example, it is basically contained in Chau-Tam-Yu \cite{CYT}.
\item[(2)] If we have the backward uniqueness of Ricci flow, we know that if $v$ is zero at some point
in $M\times [0,T)$, then $(M,g(t))$ is the trivial Ricci flow on $\R^n$.
\item[(3)] $v$ measures how far away the Ricci flow $g(t)$ is from the trivial Ricci flow on $\R^n$. So, it is
reasonable to get stronger estimates for $v$ in terms of geometry of $g(t)$.
\end{itemize}
\end{remark}
\vspace{0.5cm}

\end{document}